\documentclass{amsart} % AMS Mathe Artikel

\usepackage{amsmath,amssymb,amsthm} % Formeln, Theorem und mathematische Symbole
\usepackage{graphicx} % für Grafiken
\usepackage[arrow, matrix, curve]{xy} % für kommutative Diagramme
\usepackage{color} % Schriftfarben
\usepackage{hyperref} % für links

\DeclareMathOperator{\tb}{tb}
\DeclareMathOperator{\rot}{rot}
\DeclareMathOperator{\lk}{lk}
\DeclareMathOperator{\sign}{sign}
\DeclareMathOperator{\selfl}{sl}

\DeclareMathOperator{\divergence}{div}

\newcommand{\R}{\mathbb{R}}
\newcommand{\Z}{\mathbb{Z}}

\newcommand{\N}{\mathbb{N}}

\newcommand{\xist}{\xi_{\mathrm{st}}}

\newtheoremstyle{thm}{}{}{\itshape}{}{\bfseries}{\hfill\\}{ }{} %Thereom style
\newtheoremstyle{definition}{}{}{}{}{\bfseries}{\hfill\\}{ }{} %Definition style

\theoremstyle{thm}
\newtheorem{Theorem}{Theorem}[section]
\newtheorem{theorem}[Theorem]{Theorem}
\newtheorem{lemma}[Theorem]{Lemma}
\newtheorem{proposition}[Theorem]{Proposition}
\newtheorem{corollary}[Theorem]{Corollary}

\theoremstyle{definition}

\newtheorem{remark}[Theorem]{Remark}
\newtheorem{example}[Theorem]{Example}
\newtheorem{algorithm}[Theorem]{Algorithm}

\begingroup\expandafter\expandafter\expandafter\endgroup
\expandafter\ifx\csname pdfsuppresswarningpagegroup\endcsname\relax
\else
\fi

\begin{document}

%%%%%%%%%%%%%%%%%%%%%%%%%%%%% Titel und Author %%%%%%%%%%%%%%%%%%%%%%%%%%%%%%%%%%%%

\title{Computing rotation numbers in open books} %Titel auf Seitenüberschriften darf nicht zu lang sein.

\author{Sebastian Durst}

\address{Mathematisches Institut, Universit\"at zu K\"oln,
Weyertal 86--90, 50931 K\"oln}
\email{sdurst@math.uni-koeln.de}

\author{Marc Kegel}

\address{Mathematisches Institut, Universit\"at zu K\"oln,
Weyertal 86--90, 50931 K\"oln
\newline\url{http://www.mi.uni-koeln.de/~mkegel/} }
\email{mkegel@math.uni-koeln.de}
%%%%%%%%%%%%%%%%%%%%%%%%%%%%% Abstract %%%%%%%%%%%%%%%%%%%%%%%%%%%%%%%%%%%%

\begin{abstract}
We give explicit formulas and algorithms for the computation of the rotation number of a nullhomologous Legendrian knot on a page of a contact open book. On the way, we derive new formulas for the computation of the Thurston--Bennequin invariant of such knots and the Euler class and the $d_3$-invariant of the underlying contact structure. 
\end{abstract}

\date{\today} % Datum wird untern auf erster Seite angezeigt

 %Stichworte werden unten auf erster Seite angezeigt

 % Mathematical subject classification
%zum Beispiel:
%53D10   	Contact manifolds, general
%53D35   	Global theory of symplectic and contact manifolds
%57M10   	Covering spaces
%57M12   	Special coverings, e.g. branched
%57M25   	Knots and links in $S^3$ {For higher dimensions, see 57Q45}
%57M27   	Invariants of knots and 3-manifolds
%57M50   	Geometric structures on low-dimensional manifolds
%57N10   	Topology of general $3$-manifolds [See also 57Mxx]
%57N13   	Topology of $E^4$, $4$-manifolds
%57N15   	Topology of $E^n$, $n$-manifolds ($4 \less n \less \infty$)
%57N16   	Geometric structures on manifolds
%57N40   	Neighborhoods of submanifolds
%57R15   	Specialized structures on manifolds (spin manifolds, framed manifolds, etc.)
%57R17   	Symplectic and contact topology
%57R18   	Topology and geometry of orbifolds
%57R30   	Foliations; geometric theory
%57R65   	Surgery and handlebodies

\maketitle

%%%%%%%%%%%%%%%%%%%%%%%%%%%%%%%%%%%%%%%%%%%%%%%%%%%%%%%%%%%%%%%%%%%%%%%%%%%%%%%%
%%%%%%%%%%%%%%%%%%%%%%%%%%%%%%%%%%%%%%%%%%%%%%%%%%%%%%%%%%%%%%%%%%%%%%%%%%%%%%%%

\section{Introduction}

The \textit{classical invariants}, the \textit{Thurston--Bennequin invariant} $\tb$ and the \textit{rotation number} $\rot$, are the two most
fundamental invariants of nullhomologous Legendrian knots in contact 3-manifolds. They carry a lot of information about the contact structure (for example the contact structure is overtwisted if and only if there exists a Legendrian unknot with $\tb=0$, see~\cite{El89}) and the topological knot type (for example the classical invariants give obstructions to sliceness of a knot in $S^3$, see~\cite{Ru95}).

According to Giroux, there is a deep connection between contact manifolds and open books (cf.~\cite{Etnyre2004}). In particular, to every open book presenting a $3$-manifold there exists an (up to isotopy) unique contact structure on this $3$-manifold with contact planes arbitrarily close to the pages of the open book outside a neighbourhood of the binding.
Throughout the paper, we use the expression \emph{contact} open book to emphasize that we are in fact considering the contact $3$-manifold associated with the abstract open book in this way.

In this paper we continue in the spirit of~\cite{tb_openbooks} and consider Legendrian knots sitting on the
page of a contact open book. In~\cite{tb_openbooks} we explained how to check if such a knot is nullhomologous and if so,
how to compute its Thurston--Bennequin invariant. Here we concentrate on the second classical invariant, the rotation number, and give a formula how to compute it.

%\begin{result}[{see Theorem~\ref{thm:rot_in_ob}}]
%Let $K$ be a knot sitting on the page of an open book with
%monodromy encoded by a concatenation of Dehn twists along non-isolating
%curves.
%Then the intersection behaviour of the knot and the Dehn twist curves
%with an arc basis
%allows us to
%\begin{itemize}
%\item[(a)] decide whether $K$ is (rationally) nullhomologous,
%\item[(b1)] compute the (rational) Thurston--Bennequin invariant of $K$ if
%$K$ is (rationally) nullhomologous,
%\item[(b2)] compute the (rational) rotation number of $K$
%if $K$ is (rationally) nullhomologous,
%\item[(b3)] compute the (rational) self-linking number of a transverse push-off of $K$
%if $K$ is (rationally) nullhomologous,
%\item[(c)] compute the Poincar\'e dual of the Euler class of the contact structure,
%\item[(d)] decide if the Euler class of the contact structure is torsion and if so,
%compute its $d_3$-invariant
%\end{itemize}
%\end{result}

\begin{theorem}
\label{thm:rot_in_ob}
Let $K$ be a Legendrian knot sitting on the page of a contact open book $(\Sigma,\phi)$ with
monodromy $\phi$ given as a concatenation of Dehn twists along non-isolating
curves.
Then there exists an arc basis of $\Sigma$ such that
the intersection behaviour of $K$ and the Dehn twist curves
with the arcs give criteria and formulas to
\begin{itemize}
\item[(a)] decide whether $K$ is (rationally) nullhomologous,
\item[(b1)] compute the (rational) Thurston--Bennequin invariant of $K$
if $K$ is (rationally) nullhomologous,
\item[(b2)] compute the (rational) rotation number of $K$ if
$K$ is (rationally) nullhomologous,
\item[(b3)] compute the (rational) self-linking number of a transverse push-off of $K$
if $K$ is (rationally) nullhomologous,
\item[(c)] compute the Poincar\'e dual of the Euler class of the contact structure,
\item[(d)] decide whether the Euler class of the contact structure is torsion and if so,
compute its $d_3$-invariant
\end{itemize}
(see Algorithm~\ref{algo_everything}).
\end{theorem}

Readers only interested in using the resulting formulas can proceed to
Algorithm~\ref{algo_everything} and the following discussion in
Section~\ref{section:algorithm}.

Several results in this direction have been obtained earlier.
Etnyre and Ozbagci~\cite{EtOz08} gave a formula to compute the Euler class and the $d_3$-invariant of a contact open book, which in many cases can be
easier to compute than the one given in this paper. In order to do so, they also developed a method to calculate the rotation number of
a Legendrian knot on the page of an open book, see also the explicit calculations in~\cite{LiWa12}.
In~\cite{tb_openbooks} a method for checking if a Legendrian knot
sitting on a page is nullhomologous is developed, and if so, a formula for its
Thurston--Bennequin invariant is provided.
On the other hand, Gay and Licata~\cite{GaLi15} studied Legendrian knots in open
book which in general are not contained in a page by a generalisation of the
front projection, where it is possible to compute $\tb$ as well.

Throughout this paper, all homology groups are understood to be integral unless
indicated otherwise. We will also, by abuse of notation, use the same symbol for
a curve, the homology class and the positive Dehn twist it represents.

We will first generalise an example of~\cite{MR2557137} to compute the rotation number
of a Legendrian knot sitting on the page of a specific planar open book of
$(S^3,\xist)$.
Afterwards we use the method of~\cite{Avdek2013} to find an embedding of a more general
non-planar  abstract open book into $(S^3,\xist)$ and give formulas for computing the rotation number in these cases.

For the general case, we first use Avdek's algorithm~\cite{Avdek2013} for transforming a contact open book into a contact surgery diagram along a Legendrian link
and then compute the invariants from the resulting contact surgery diagram
via~\cite{rot_surgery}, which builds on~\cite{MR2557137,GeOn15,Conway2014,Kegel2016}.

We begin with an application of our results on the binding number of Legendrian knots, which we propose to study in analogy to the binding number of a contact
manifold as introduced in~\cite{EtOz08}.

%%%%%%%%%%%%%%%%%%%%%%%%%%%%%%%%%%%%%%%%%%%%%%%%%%%%%%%%%%%%%%%%%%%%%%%%%%%%%%%%
\subsection*{Application to the binding number of Legendrian knots}\hfill
\label{section:application_binding number}

Let $K$ be a Legendrian knot in a contact $3$-manifold $(M,\xi)$. Then the support genus $\textrm{sg}(K)$
is defined to be the minimal genus of the page of a contact open book decomposition of $(M,\xi)$
in which $K$ is contained in a single page, i.e.\
$$
\textrm{sg}(K) = \min\big\{ g(\Sigma) \,\mid\, K\subset\Sigma \big\},
$$
where $g(\Sigma)$ is the genus of the surface $\Sigma$ (see~\cite{Ona2010}).

In analogy to the binding number of a contact manifold as introduced by~\cite{EtOz08},
we propose to define
the binding number $\textrm{bn}$ of $K$ to be the minimal number of boundary components
of the pages of contact open book decompositions with minimal genus containing $K$ in a page, i.e.\
$$
\textrm{bn}(K) := \min\big\{ |\partial\Sigma|\colon\thinspace K\subset\Sigma \textrm{ with }g(\Sigma)
=sg(K)\big\}.
$$

\begin{corollary}
Let $K$ be a Legendrian knot with non-vanishing rotation number and support genus
$\textrm{sg}(K) = 1$ in an arbitrary contact $3$-manifold.
Then the binding number of $K$ is at least two.
\end{corollary}

Note that this also becomes evident from the proof of Lemma~6.1 in~\cite{EtOz08}.

\begin{proof}\hfill\\
Suppose that $K$ has support genus and binding number both equal to one,
then one can easily check using Theorem~\ref{thm:rot_in_ob} or via the explicit formulas
given in Algorithm~\ref{algo_everything} that the rotation number of $K$ vanishes.
\end{proof}

\begin{example}\label{ex:nonempty}
It is known that all Legendrian realizations of torus knots $T_{2,2n+1}$, $n\in\N$, with Thurston--Bennequin invariant at least one and non-vanishing rotation number have support genus equal to one (see Theorem~1.3 in\cite{LiWa12}) and thus binding number at least two.
	\end{example}

%%%%%%%%%%%%%%%%%%%%%%%%%%%%%%%%%%%%%%%%%%%%%%%%%%%%%%%%%%%%%%%%%%%%%%%%%%%%%%%%
\section{Background: Legendrian curves on open books}
For the basics in low-dimensional contact topology we refer the reader for example to~\cite{Etnyre2004,OzSt04,Geiges2008}.
Nevertheless, we will briefly recall some well-known facts about which curves sitting on the page of
a contact open book represent Legendrian knots.

Let $L$ be a simple closed curve on a convex surface $S$.
We call $L$ \textbf{non-isolating} if every component of $S\setminus L$
has non-empty intersection with the dividing set $\Gamma$ of~$S$.

\begin{lemma}
A simple closed curve $L$ on $S$ represents a Legendrian knot (i.e.\ can be realised as a Legendrian knot by a small perturbation of $S$ through convex surfaces) if and only
if $L$ is non-isolating.
\end{lemma}

\begin{proof}
A non-isolating simple closed curve always represents a Legendrian knot by the Legendrian realisation
principle (see \cite[Theorem~3.7]{Honda2000}).
So let $L$ be not non-isolating, i.e.\ there is a component
$S_0$ of $S\setminus L$ with $S_0\cap \Gamma = \emptyset$,
and assume that $L$ represents a Legendrian knot.
%Then $\tb(L) = 0$ as the contact
%framing of $L$ coincides with the Seifert framing given by $\overline{S}_0$.
Without loss of generality, we have $\divergence_\Omega(X) > 0$ on
$\overline{S}_0$, where $\Omega$ is a volume form on $S$ and $X$ the vector
field defining the characteristic foliation.
Hence,
$$
0 < \int_{\overline{S}_0}{\divergence_\Omega(X)} =
\int_{\overline{S}_0}{d(i_X\Omega)}
= \int_L{i_X\Omega} = \int_L{\alpha} = 0,
$$
where $\alpha$ denotes the contact form and the last equality holds because
L is Legendrian.
\end{proof}

Here we are interested in the special case of the page $\Sigma$ of an open book, which is convex
with $\Gamma = \partial\Sigma$.
In particular, if $\partial\Sigma$ is connected, $L$
represents a Legendrian knot if and only if $L$ is non-separating.
Note also, that for every Legendrian link in a contact manifold there exists
a compatible open book decomposition such that the link is contained in a page
(cf.\ \cite{AkbulutOz} or \cite[Corollary~4.23]{Etnyre2004}), i.e.\ our assumption of a Legendrian
knot sitting on the page of an open book is not exotic at all.

%%%%%%%%%%%%%%%%%%%%%%%%%%%%%%%%%%%%%%%%%%%%%%%%%%%%%%%%%%%%%%%%%%%%%%%%%%%%%%%%

\section{A special planar case}
We begin by discussing a method to compute the rotation number in an easy planar case
which is based on the idea presented in \cite[Lemma~4.1]{MR2557137}.

Suppose that $\Sigma$ is \emph{planar}, i.e.\ $\Sigma$ is a disc with $k$ holes
$$
\Sigma \cong D^2 \setminus \left(\bigsqcup_{i=1}^k D^2_i\right),
$$
and the monodromy is given by $\phi = \beta_k^{+1} \circ \cdots \circ \beta_1^{+1}$,
where $\beta_i^{+1}$ denotes a positive Dehn twist along a curve $\beta_i$ parallel
to the inner boundary $\partial D^2_i$.
We furthermore assume that the curves $\beta_i$ are oriented
consistently with the boundary orientation induced by $\Sigma$
(see Figure~\ref{fig:diskWithHoles}).
In particular, by destabilising the open book, we see that $(\Sigma, \phi)$
describes the standard contact $3$-sphere
$(S^3, \xi_\text{st})$ and from this it also follows that every Legendrian knot $K$ sitting on the page of this contact open book
is some Legendrian unknot.

\begin{figure}[htb] 
\centering
\def\svgwidth{0.5\columnwidth}
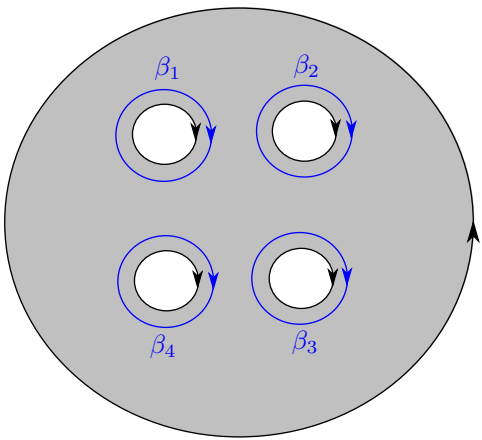
\caption{A planar open book decomposition of $(S^3,\xist)$}
\label{fig:diskWithHoles}
\end{figure}

\begin{proposition}
\label{prop:planar}
Let $K$ be a Legendrian knot sitting on the page of a planar open book
$(\Sigma, \phi)$ with $\phi$ as described above.
Then the following holds:

\begin{enumerate}
\item $K = \sum_{i=1}^{k} b_i \beta_i \in H_1(\Sigma)$ such that either all $b_i\in\{+1,0\}$ or
all $b_i\in \{-1,0\}$,
\item the rotation number of $K$ computes as
$$
\rot (K) = \sum_{i=1}^{k} b_i - \sign\left(\sum_{i=1}^{k}b_i\right).
$$
\end{enumerate}
\end{proposition}

\begin{proof}\hfill\\
(1)
First note, that a simple closed curve cannot have $|b_i| > 1$ or it would
have self-intersections. With orientations chosen as above, one also observes
that all non-vanishing $b_i$ have to be equal.\\
(2)
By the first statement, we can glue small oriented rectangular bands connecting
the $b_i\beta_i$ with non-vanishing coefficients $b_i$ inside $\Sigma$
in such a way that the oriented boundary of the
resulting region is isotopic to $K$ in $\Sigma$
(cf.\ Figure~\ref{fig:PlanarExample}).
The orientation of these rectangles coincides with the orientation of the page
$\Sigma$ exactly if the $b_i$ are positive.

Note that the $\beta_i$ are unknots with Thurston--Bennequin invariant $-1$ and
vanishing rotation number.
Indeed, $\beta_i$ can be assumed to be parallel to a Dehn twist curve arising by
a stabilisation. These curves bound a disc in the complement and by the Dehn twist,
the Seifert framing differs by one from the contact framing given by the page.
So $\beta_i$ is a $\tb=-1$ unknot, i.e.\ the rotation number is zero.
Furthermore, a Seifert surface for $K$ is given by the union of the discs bounded
by the non-vanishing $b_i\beta_i$ (in the complement of the page)
and the attached bands in the page.
The rotation number computes as the sum of the indices of a vector field in the
contact structure extending the positive tangent of $K$ over $\Sigma$.
As $\rot(\beta_i) = 0$, an extension without zeros is possible over the discs
bounded by $\beta_i$ and we only have to study the bands.
As the contact framing and the page framing coincide, this reduces the problem
to extending the positive tangent vector field to the boundary of the bands over
the bands in $\Sigma$. This is $\pm 1$ for each band by Poincar\'e--Hopf,
depending on whether the orientation of the band agrees with the orientation of
the page $\Sigma$ or not.
Hence, the rotation number of $K$ is a signed count of the number of bands, i.e.\
$\rot (K) = \sum_{i=1}^{k} b_i - \sign\big(\sum_{i=1}^{k}b_i\big)$.
\end{proof}

\begin{figure}[htb] 
\centering
\def\svgwidth{0.55\columnwidth}
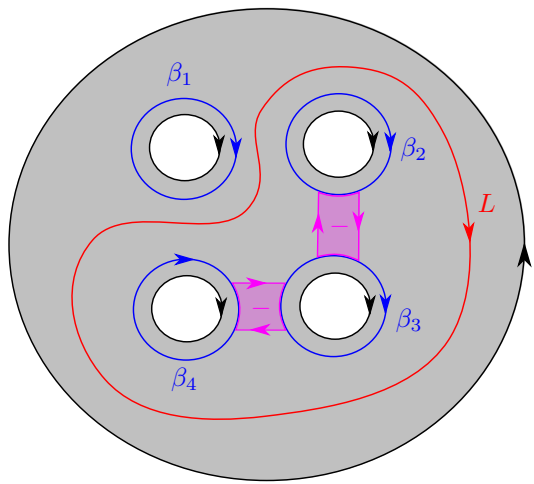
\caption{A Legendrian knot on the page of a planar open book of $(S^3,\xist)$}
\label{fig:PlanarExample}
\end{figure}

\begin{remark}
The formula from Proposition~\ref{prop:planar} can also be obtained by observing
that a curve enclosing $k$-holes is the result of $(k-1)$-times stabilising a curve running
around a single hole. The latter has Thurston--Bennequin invariant
$-1$ and vanishing rotation number.
\end{remark}

\begin{example}
\label{ex:planar}
Consider the Legendrian knot $L$ on the planar open book of $(S^3,\xist)$ as depicted
in Figure~\ref{fig:PlanarExample}.
The class in the first homology group of $\Sigma$ represented by $L$ can be written as 
$$
L = \sum_{i=1}^{4} b_i\beta_i = \beta_2 +\beta_3 +\beta_4.
$$
By Proposition~\ref{prop:planar}, the rotation number of $L$ is
$$
\rot(L)
= \sum_{i=1}^{4} b_i - \sign\left(\sum_{i=1}^{k}b_i\right)
= 2.
$$
\end{example}

This method is not known to generalise to non-planar open books. One reason is,
that on surfaces of higher genus, the isotopy class of a curve is not determined
by its homology class.

%%%%%%%%%%%%%%%%%%%%%%%%%%%%%%%%%%%%%%%%%%%%%%%%%%%%%%%%%%%%%%%%%%%%%%%%%%%%%%%%

\section{Another special case}

Next we consider knots on open books $(\Sigma, \phi)$ of the
standard contact $3$-sphere with an arbitrary page but a special monodromy.
Denote the genus of $\Sigma$ by $g$ and the number of boundary components by $h+1$.
Suppose that the monodromy is given by
$$
\phi = \beta_{g+h}^{+1} \circ\cdots\circ \beta_{g+1}^{+1}\circ \beta_g^{+1} \circ \alpha_g^{+1} \circ\cdots\circ
\beta_1^{+1} \circ \alpha_1^{+1}
$$
as indicated in Figure~\ref{fig:AbstractOpenBookWithBoundary}. We also choose orientations
of $\alpha_i$ and $\beta_i$ as in
the picture. In particular, the signed count $\alpha_i \bullet \beta_j$ of intersection points between $\alpha_i$ and $\beta_j$ is $\delta_{ij}$.
Let $r_i$, $i=1,\ldots,g+h-1$, be the depicted reducing arcs, which do not intersect
the $\alpha$- and $\beta$-curves, i.e.\ when cutting
along them the
page $\Sigma$ decomposes into a collection of tori
 with a disc removed and annuli.
Let $a_i$ and $b_i$ be arcs on the page $\Sigma$ representing a basis
of $H_1(\Sigma,\partial\Sigma)$ dual to $\{\alpha_i,\beta_i\}$ with respect to
the intersection product (oriented such that $\alpha_i\bullet a_i = 1$, $\beta_j\bullet b_j=1)$.

\begin{figure}[htb] 
\centering
\def\svgwidth{\columnwidth}
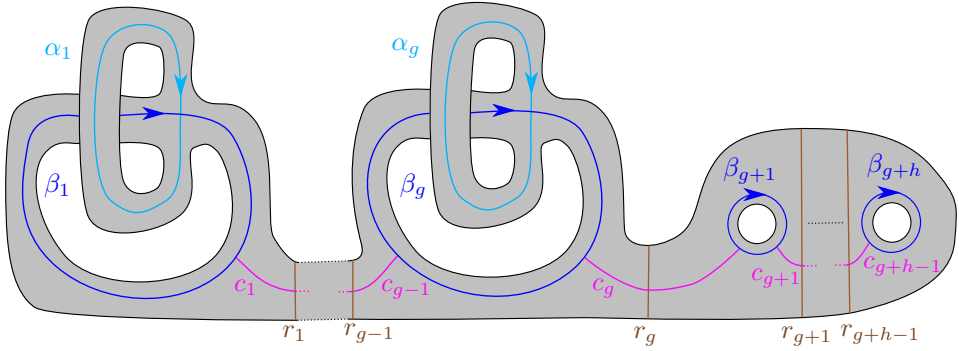
\caption{A non-planar open book of $(S^3,\xist)$ with arbitrarily many boundary components}
\label{fig:AbstractOpenBookWithBoundary}
\end{figure}

The following algorithm will be applied to a word corresponding to the knot $K$
in Proposition~\ref{prop:special_s3}.
Note that the conventions presented below for labelling vertical tangencies in this setting by $\rho_+$ and $\lambda_+$
do not agree with those for counting cusps of a Legendrian front
projection as in \cite[Proposition~3.5.19]{Geiges2008}.

\begin{algorithm}
\label{algo:c_from_word}
Let $w$ be a word in $\langle\alpha_i, \beta_i \mid i=1\ldots, k\rangle$.
Set $\lambda_+$ to be the number of times a $\beta^{-1}$ is followed by an
$\alpha^{-1}$ of the same index also considering the step from the last to the first letter,
and similarly, set $\rho_+$ equal to the number of times an $\alpha^{-1}$ is
followed by a $\beta^{-1}$ of the same index.

Denote places where the index changes by $r_u$ ($r_d$) if the index
increases (decreases) -- including the last position if the index of the last letter is not equal to the index of the first letter.
For instance, in the word
$$
\alpha_1 \beta_2 \alpha_2 \beta_4^{-1} \alpha_3^{-1}\beta_2
$$
we have five positions of index changes:
$$
\alpha_1 r_u \beta_2 \alpha_2 r_u \beta_4^{-1} r_d \alpha_3^{-1} r_d \beta_2 r_d.
$$

Now run through the index changes and increase $\lambda_+$ and $\rho_+$
according to the following rule:
\begin{itemize}
\item increase $\lambda_+$ by $1$ for
\begin{itemize}
\item a $\beta^{-1}$ followed by $r_u$
\item $r_d$ followed by an $\alpha^{-1}$
\end{itemize}
\item increase $\rho_+$ by $1$ for
\begin{itemize}
\item an $\alpha^{-1}$ followed by $r_d$
\item a $\beta$ followed by $r_d$.
\end{itemize}
\end{itemize}
In the example sequence above, we have $\lambda_+ = 0=\rho_+$ as staring values and then get
$\lambda_+ = 1$ and $\rho_+ = 2$ as the final result after following the rules for the increments.
\end{algorithm}

\begin{proposition}
\label{prop:special_s3}
Let $K$ be an oriented non-isolating knot on the abstract open book $(\Sigma, \phi)$
of $(S^3,\xist)$ specified above.
Choose a starting point on $K$ and write $K$ as a word in the $\alpha_i$ and $\beta_i$
by noting intersections with $a_i$ and $b_i$ when traversing along $K$.
Then the rotation number of $K$ is
$$
\rot (K) = \rho_+ - \lambda_+
$$
with $\rho_+$ and $\lambda_+$ calculated from the presentation of $K$ as described in
Algorithm~\ref{algo:c_from_word}.
\end{proposition}

\begin{proof}\hfill\\
First note that without loss of generality, we can assume that the page $\Sigma$
has only a single boundary component by stabilising the open book along arcs not
intersecting the $r_i$ connecting a hole to the outer boundary component.
Then the open book $(\Sigma, \phi)$ can be embedded into $(S^3, \xi_\text{st})$
with the front projection shown in Figure~\ref{fig:embeddedOpenBook}
(in lightly shaded regions the orientation of $\Sigma$ agrees with the blackboard orientation, in darkly shaded regions the orientations disagree)
-- the embedded page $\Sigma$ is the ribbon of the Legendrian graph
displayed in the upper half of Figure~\ref{fig:embeddedOpenBook}
(see \cite{Avdek2013} for details). Note that in particular, the contact vector field
$\partial_z$ is transverse to the embedded page.
Furthermore, after rescaling the embedding can be assumed to be such that
in $\R^3\subset S^3$ we have
\begin{align*}
[-1,1]\times \Sigma &\longrightarrow \big(\R^3, \xi_\textrm{st} = \ker(xdy + dz)\big),\\
(t,p)&\longmapsto p + (0,0,t),
\end{align*}
i.e.\ for a point $p$ in the interior of the embedded page $\{0\}\times \Sigma$ the line through $p$ parallel to the $z$-axis hits every page in $[-1,1]\times \Sigma$ exactly once. So we can relate to a specific page in $[-1,1]\times \Sigma$ by its shift in the $z$-direction.

\begin{figure}[htb] 
\centering
\def\svgwidth{\columnwidth}
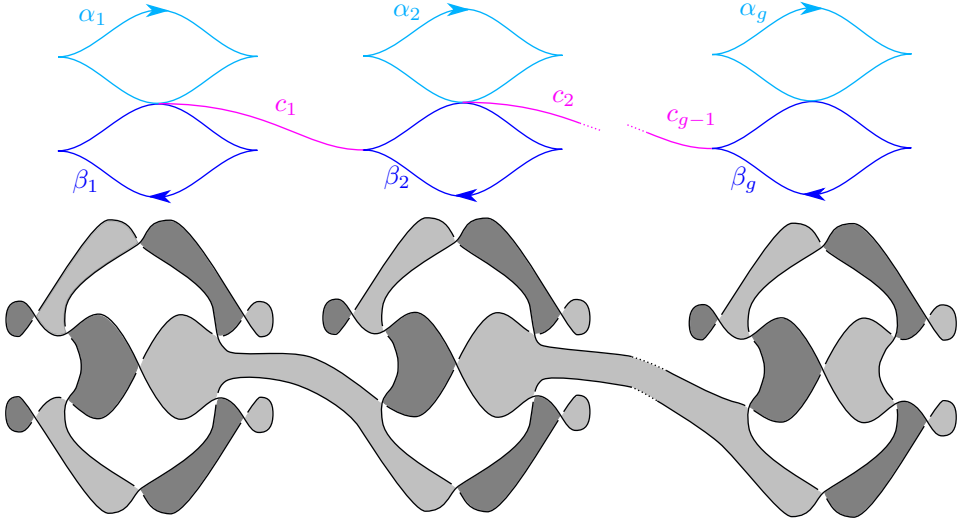
\caption{An embedding into $(S^3,\xist)$ of the (stabilised) abstract open book from Figure~\ref{fig:AbstractOpenBookWithBoundary} and the Legendrian graph shown in the front projection}
\label{fig:embeddedOpenBook}
\end{figure}

The rotation number of a nullhomologous Legendrian knot with respect to a Seifert surface $S$
is given by the rotation of its tangent vector
with respect to a fixed trivialisation of the contact planes over $S$.
If the contact structure is globally trivialisable,
one can instead fix a global trivialisation.
The standard contact structure $\xist$ on $\R^3\subset (S^3,\xist)$ can be trivialised globally by
$\partial_x$ and $\partial_y - x\partial_z$.
As the contact vector field $\partial_z$ is transverse to the embedded page $\Sigma$
of the open book, this trivialisation also induces a trivialisation of the tangent
planes to $\Sigma$.
Then the rotation number of the Legendrian realisation of a curve sitting on the
page agrees with the rotation of the original curve on the page with respect to the induced trivialisation.

The projection of $\partial_x$ to $\Sigma$ along $\partial_z$ lies in the
$xz$-plane.
Observe that the $\partial_z$-component changes sign when passing from a lightly
shaded region to a darkly shaded region and vice-versa.
To compute the rotation of a curve on the embedded page
which is non-singular in the front projection diagram, we thus have to count
vertical tangencies in the front projection
according to the rule described in Figure~\ref{fig:lambdaRho}.
The rotation then equals $\rho_+ - \lambda_+$.
Alternatively, we can also compute it as $\lambda_- - \rho_-$ where $\lambda_-$ and $\rho_-$ are defined analogously to $\lambda_+$ and $\rho_+$.

\begin{figure}[htb] 
\centering
\def\svgwidth{\columnwidth}
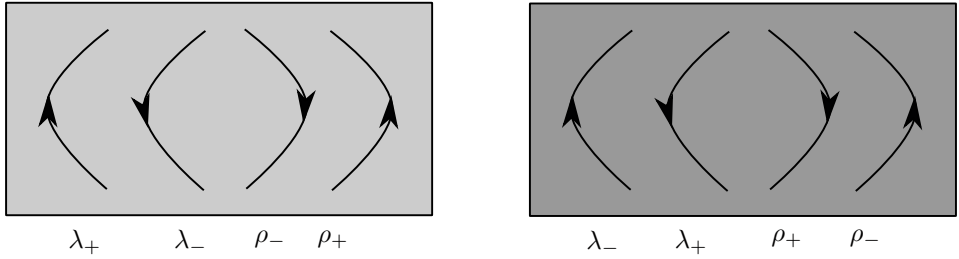
\caption{The labelling of the vertical tangencies.}
\label{fig:lambdaRho}
\end{figure}

In fact, we do not even have to count all vertical tangencies,
but we can ignore those cancelling each other.
To this end, we write $K$ as a word in the $\alpha_i$ and $\beta_i$
by noting intersections with $a_i$ and $b_i$ when traversing along $K$.
Observe that the $\alpha$- and $\beta$-curves have vanishing rotation, as they
have two vertical tangencies cancelling each other.
Changing from $\alpha_i$ to $\beta_i$ accounts for a $\lambda_-$,
changing from $\beta_i$ to $\alpha_i$ for a $\rho_-$.
Likewise, the change from $\alpha_i^{-1}$ to $\beta_i^{-1}$ gives a $\rho_+$,
the one from $\beta_i^{-1}$ to $\alpha_i^{-1}$ a $\lambda_+$.
It is easily verified that all other changes with fixed index do not introduce
vertical tangencies.
In particular, a knot not intersecting any of the reducing arcs
has vanishing rotation number, since it has $\lambda_+ = \rho_+$.
It thus remains to inspect those tangencies occurring before or after an intersection
with a reducing arc.
These intersections happen when the index of the letters change. The vertical
tangencies occurring in these cases are summarised in Table~1.

%\begin{center}
\begin{table}[h]
\begin{tabular}{c|c}
leaving to the right from & count \\
\hline
$\alpha$ & $\lambda_-$\\
$\alpha^{-1}$ & --\\
$\beta$ & --\\
$\beta^{-1}$ & $\lambda_+$\\
\hline
\hline
coming from the left to & count \\
\hline
$\alpha$ & $\rho_-$ \\
$\alpha^{-1}$ & --\\
$\beta$ & --\\
$\beta^{-1}$ & $\rho_+$\\
\hline
\hline
leaving to the left from & count \\
\hline
$\alpha$ & --\\
$\alpha^{-1}$ & $\rho_+$\\
$\beta$ & $\rho_+$\\
$\beta^{-1}$ & --\\
\hline
\hline
coming from the right to & count \\
\hline
$\alpha$ & --\\
$\alpha^{-1}$ & $\lambda_+$\\
$\beta$ & $\lambda_-$\\
$\beta^{-1}$ & --\\
\hline
\end{tabular}
\vspace{0.2cm}
\caption{}
\end{table}
%\end{center}

Hence,
the rotation number can be computed from the word according to the rule
given in Algorithm~\ref{algo:c_from_word}.
\end{proof}

\begin{example}
Consider the knot on the embedded page of the open book of $(S^3,\xist)$
 given in Figure~\ref{fig:example_S3}.
The knot corresponds to the word
$\alpha_1 \beta_2 \alpha_2 \alpha_3^{-1} \beta_3 \beta_2$.
The vertical tangencies corresponding to the $\alpha$- and $\beta$-curves which
immediately cancel are marked in green.
The remaining vertical tangencies are marked blue and labelled.
We have
$\rho_+ = 2$, $\lambda_+ = 0$, $\rho_- = 1$, $\lambda_- = 3$,
i.e.\
the rotation number of the Legendrian knot represented by $K$ is
$$
\rot (K) = \rho_+ - \lambda_+  = \lambda_- - \rho_- = 2.
$$

\begin{figure}[htb] 
\centering
\def\svgwidth{\columnwidth}
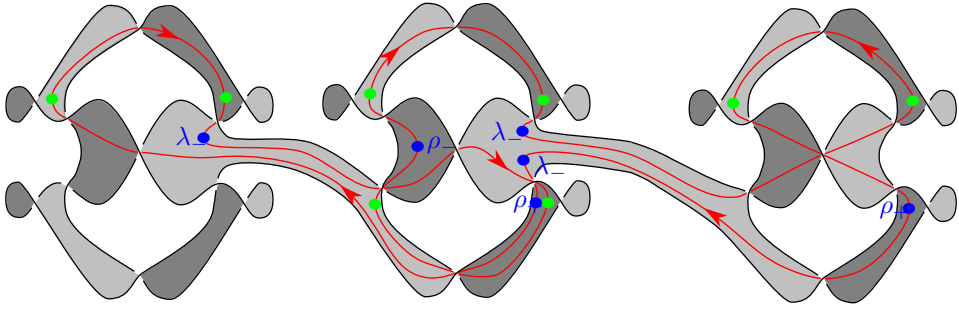
\caption{A knot on an embedded page in $(S^3,\xist)$. Vertical tangencies cancelling
each other are marked green, other vertical tangencies are marked blue and labelled.}
\label{fig:example_S3}
\end{figure}

We will now apply Algorithm~\ref{algo:c_from_word} on the word
$\alpha_1 \beta_2 \alpha_2 \alpha_3^{-1} \beta_3 \beta_2$.

As neither a $\beta^{-1}$ is followed by an $\alpha^{-1}$ of the same index,
nor an $\alpha^{-1}$ by a $\beta^{-1}$, we set $\lambda_+ = 0 = \rho_+$ as initial values.
Next, we consider the index changes:
$$\alpha_1 r_u \beta_2 \alpha_2 r_u \alpha_3^{-1} \beta_3 r_d \beta_2 r_d.$$
The positions $\beta_3 r_d$ and $\beta_2 r_d$ both increase $\rho_+$ by one,
all other positions leave the counts unchanged.
Hence, the algorithm yields
$$
\rot (K) = \rho_+ - \lambda_+ = 2.
$$
Note that we could also adapt the rules specified in the algorithm to consider
$\rho_-$ and $\lambda_-$ instead using the proof of the preceding proposition.
\end{example}

%%%%%%%%%%%%%%%%%%%%%%%%%%%%%%%%%%%%%%%%%%%%%%%%%%%%%%%%%%%%%%%%%%%%%%%%%%%%%%%%
%%%%%%%%%%%%%%%%%%%%%%%%%%%%%%%%%%%%%%%%%%%%%%%%%%%%%%%%%%%%%%%%%%%%%%%%%%%%%%%%

\section{The general case}
\label{section:general_case}

Now we are prepared to deal with a Legendrian knot in a general contact open book.
The idea is to change the open book to the special case discussed in the
previous section by a sequence of surgeries, then compute the rotation number
in $(S^3,\xist)$ as above and finally use \cite{rot_surgery} with the
inverse surgeries to get the rotation number of the Legendrian in the original
open book. The result will be presented in a formula that can be directly
computed with the data of the original open book.

In the following remark, we will briefly recall how to compute the rotation number
in contact surgery diagrams.
%%%%%%%%%%%%%%%%%%%%%%%%%%%%%%%%%%%%%%%%%%%%%%%%%%%%%%%%%%%%%%%%%%%%%%%%%%%%%%%%
\begin{remark}[Computing $\rot$ in a surgery diagram (see \cite{rot_surgery})]
\label{rem:rot_surgery}
For an oriented Legendrian link $L = L_1 \sqcup \ldots \sqcup L_k$ in
$(S^3, \xist)$
let $(M, \xi)$ be the contact manifold obtained from $(S^3,\xist)$ by contact $({1}/{n_i})$-surgeries ($n_i \in \Z$) along $L_i$.
Denote the topological surgery coefficients by ${p_i}/{q_i}$,
i.e.\
$$
\frac{p_i}{q_i} = \frac{n_i\tb(L_i) + 1}{n_i}.
$$
Let $L_0$ be an oriented Legendrian knot in the complement of $L$ and define
the vector $\mathbf{l}$ with components $l_i = l_{0i}$
and the generalised linking matrix
$$
Q = \begin{pmatrix} p_1 & q_2 l_{12} & \cdots & q_n l_{1k} \\
q_1 l_{21} & p_2 & & \\
\vdots & & \ddots & \\
q_1 l_{k1} & & & p_k \end{pmatrix},
$$
where $l_{ij} := \lk (L_i, L_j)$.
The knot $L_0$ is (rationally) nullhomologous in $M$ if and only if there is an integral (rational) solution
$\mathbf{a}$ of the equation $\mathbf{l} = Q\mathbf{a}$, 
in which case its (rational) rotation number in $(M, \xi)$ with respect to the Seifert class
$\widehat{\Sigma}$ constructed in \cite{rot_surgery} is equal to
$$
\rot_M (L_0, \widehat{\Sigma}) = \rot_{S^3} (L_0) - \sum_{i=1}^k{a_i n_i \rot_{S^3} (L_i)}.
$$
\end{remark}
%%%%%%%%%%%%%%%%%%%%%%%%%%%%%%%%%%%%%%%%%%%%%%%%%%%%%%%%%%%%%%%%%%%%%%%%%%%%%%%%

\begin{proof}[{Proof of Theorem~\ref{thm:rot_in_ob}}]\hfill\\
Let $K\subset (M, \xi)$ be a Legendrian knot sitting on the page of a compatible
open book
$$
(\Sigma, \phi = T_l^{\pm n_l} \circ\cdots\circ T_1^{\pm n_1})
$$
with monodromy encoded in a concatenation of Dehn twists,
where $T^{\pm n}$ denotes $n$ positive (resp.\ negative) Dehn twists along the
non-isolating oriented curve $T$ ($n\in \N$).
We denote the genus of $\Sigma$ by $g$ and the number of boundary components by $h+1$.

In the following, we want to choose a \emph{special} arc basis of $\Sigma$ to
exactly mimic the setting from Proposition~\ref{prop:special_s3} (also see Remark~\ref{rem:special_arcs}).
Together with a suitable monodromy yielding $(S^3,\xist)$, this will enable us to use the proposition to compute the invariants first in $(S^3,\xist)$ and then to apply the surgery formulas to obtain the desired result.

Choose reducing arcs $r_1,\ldots, r_{g+h-1}$ such that
when cutting along $r_i$
\begin{itemize}
\item
$\Sigma$ decomposes into a surface $\Sigma_i$ of genus $i$ containing $r_1,\ldots, r_{i-1}$ with one boundary component
and a surface of genus $g-i$ with $h+1$ boundary components
for $i=1,\ldots,g$,
\item
$\Sigma$ decomposes into a surface $\Sigma_i$ of genus $g$ containing $r_1,\ldots, r_{i-1}$ with $i+1$ boundary components
and a disk with $h-i$ holes
for $i=g+1,\ldots,g+h-1$.
\end{itemize}

Then choose an arc basis of $\Sigma_i\setminus\Sigma_{i-1}$. Orient and label the arcs by $a_i, b_i$ such that when travelling along the oriented boundary of $\Sigma$ from 
\begin{itemize}
\item $r_1$ to $r_1$
\begin{itemize}
\item first $a_1$ is met pointing outwards, then $b_1$ is met pointing inwards
if $g\geq 1$
\item
$b_1$ is met and pointing outwards if $g=0$
\end{itemize}
\item $r_{i-1}$ to $r_i$ only $b_i$ is met and pointing outwards ($i=2,\ldots,g+h-2$)
\item
$r_{g+h-1}$ to $r_{g+h-1}$
\begin{itemize}
\item first $b_g$ is met pointing outwards, then $a_g$ is met pointing outwards
if $h=0$
\item $b_{g+h}$ is met and pointing outwards if $h>0$.
\end{itemize}
\end{itemize}

Choose non-trivial oriented simple closed curves
$\alpha_i,\beta_i$ representing a basis of $H_1(\Sigma)$ dual to the arcs
with respect to the intersection product on $\Sigma$
oriented such that $\alpha_i\bullet a_i = 1$, $\beta_j\bullet b_j=1$ and
$\alpha_i\bullet\beta_i = 1$
(i.e.\ the situation is as in Figure~\ref{fig:AbstractOpenBookWithBoundary}).

\begin{remark}
\label{rem:special_arcs}
Note that the arc basis cannot be chosen arbitrarily, as we will use it to write
the knot as a word in $\alpha_i, \beta_i$ as above and use the formula from
Proposition~\ref{prop:special_s3} to compute the rotation from this word.
For this to work with the given formula, we have to ensure that the word we get
in the abstract setting is the same as the word we get in the embedded case,
which coincides with the specific abstract open book depicted in
Figure~\ref{fig:AbstractOpenBookWithBoundary}.
In particular, the word obtained from the oriented boundary of the page is
$$
\alpha_1^{-1} \beta_1 \alpha_1 \beta_1^{-1} \beta_2^{-1}\cdots
\beta_{g+h}^{-1} \alpha_g^{-1}\beta_g\alpha_g \cdots
\alpha_2^{-1}\beta_2\alpha_2.
$$
A different arc basis would require a different formula to compute the rotation
number from the word, see also
Example~\ref{ex:planar_word}.

This is only important for calculating the rotation number in $(S^3,\xist)$
which is not determined by the class of the knot in the homology of the page --
the linking information required to compute the rotation number via the
surgery formula is purely homological and does not depend on the specific ordering.
In particular, we can use an arbitrary arc basis in a planar open book if we
use Proposition~\ref{prop:planar} to compute the rotation number of the involved curves in
$(S^3,\xist)$.
\end{remark}

Observe that we can get from the open book
$$
(\Sigma, \phi_{S^3} = 
\beta_{g+h}^{+1} \circ\cdots\circ \beta_{g+1}^{+1}\circ
\beta_g^{+1} \circ \alpha_g^{+1} \circ\cdots\circ
\beta_1^{+1} \circ \alpha_1^{+1})
$$
to the open book
$(\Sigma, \phi)$
by a sequence of contact surgeries along Legendrian knots corresponding to the
Dehn twist curves.

By the algorithm presented in \cite{Avdek2013}, the surgery link is as follows:
every component corresponding to a Dehn twist sits on a page of the embedded open
book, the shift in $z$-direction of the respective page relates to the position of the Dehn
twists in the monodromy factorisation -- the later the Dehn twist is performed, the higher the level of the page.
Using Avdek's convention,
we will denote a knot $K$ sitting on the page with level $t$ by $K(t)$.

Observe that the $\alpha_i(s), \beta_i(s)$ are unknots with rotation number zero
and Thurs\-ton--Bennequin invariant $-1$
and that for $t\neq s$ we have
$$
\arraycolsep=5pt\def\arraystretch{1.6}
\begin{array}{rl}
\lk\big(\alpha_i(t), \beta_j(s)\big) = &
\begin{cases}
0, & \textrm{if } i\neq j \textrm{ or }t>s,\\
-1, & \textrm{if } i= j \textrm{ and }t<s,
\end{cases}\\[0.5cm]
\lk\big(\alpha_i(t), \alpha_j(s)\big) = &
\begin{cases}
0, & \textrm{if } i\neq j,\\
-1, & \textrm{if } i = j,
\end{cases}\\[0.5cm]
\lk\big(\beta_i(t), \beta_j(s)\big) = &
\begin{cases}
0, & \textrm{if } i\neq j,\\
-1, & \textrm{if } i = j.
\end{cases}
\end{array}
$$
If $s=t$, the curves form a Legendrian graph on a single page with $\alpha$ and $\beta$ joined by a single transverse intersection point.

The first homology class represented by a knot $K$ on $\Sigma$ can then be written as
$$
K = \sum_{i=1}^{g+h}\big((K\bullet a_i) \alpha_i + (K\bullet b_i) \beta_i)\big)
$$
and hence
$$
K(t) = \sum_{i=1}^{g+h}\big((K\bullet a_i) \alpha_i(t) + (K\bullet b_i) \beta_i(t)\big),
$$
where $(K\bullet a_i)$ is defined to be zero for $k>g$.
The linking number of two knots $K_1(t)$ and $K_2(s)$ behaves linear and distributive
with respect to this decomposition, i.e.\ the linking number is easily computable with
the linking behaviour of the $\alpha$ and $\beta$ curves specified above.

It is well-known that the surgery link in $S^3$ to obtain
$(\Sigma, \phi)$ is the link $L = L_1 \sqcup \ldots \sqcup L_{2g+h+l}$ as
specified in Table~2 (e.g.\ see Avdek's algorithm~\cite{Avdek2013}).

\begin{table}[h]
$$
\begin{array}{c|c|c}
\textrm{name} & \textrm{knot} & \textrm{contact surgery coefficient}\\
\hline
%%%%%
L_1 & \beta_1(-1) & +1\\
\vdots & \vdots & \vdots\\
L_{g+h} & \beta_{g+h}(-1) & +1\\
%%%%%
L_{g+h+1} & \alpha_{1}(0) & +1\\
\vdots & \vdots & \vdots\\
L_{2g+h} & \alpha_{g}(0) & +1\\
%%%%%
L_{2g+h+1} & T_{1}(1/l) & {\mp 1}/{n_1} \\
\vdots & \vdots & \vdots\\
L_{2g+h+l} & T_{l}(l/l) & {\mp 1}/{n_l}\\
\end{array}
$$
\vspace{0.2cm}
\caption{}
\end{table}

To compute the rotation number of a knot on a page of $(\Sigma, \phi)$ using the
method explained in Remark~\ref{rem:rot_surgery},
we need the generalised linking matrix $Q$ -- which requires us to know $\tb$
for deducing the topological surgery coefficient from the contact one as well
as all linking numbers -- and the rotation numbers in $(S^3, \xist)$.

%$L_i = \beta_i(-1)$, $i=1,\ldots,g+h$
%
%$L_i = \alpha_{j}(0)$, $j=i-(g+h)$, $i=g+h+1,\ldots,2g+h$
%
%$L_i = T_{j}(j/l)$, $j=i-(2g+h)$, $i=2g+h+1,\ldots,2g+h+l$

%$$
%\begin{array}{c|c|c|c|c|c}
%\textrm{name} & \textrm{knot} & \textrm{contact coefficient coefficient} &
%\tb_{S_3} & p/q & \rot_{S^3}\\
%\hline
%%%%%%
%L_1 & \beta_1(-1) & +1 & -1 & 0/1 & 0\\
%\vdots & \vdots & \vdots& \vdots & \vdots & \vdots \\
%L_{g+h} & \beta_{g+h}(-1) & +1 & -1 & 0/1 & 0 \\
%%%%%%
%L_{g+h+1} & \alpha_{1}(0) & +1 & -1 & 0/1 & 0 \\
%\vdots & \vdots & \vdots& \vdots & \vdots & \vdots \\
%L_{2g+h} & \alpha_{g}(0) & +1 & -1 & 0/1 & 0 \\
%%%%%%
%L_{2g+h+1} & T_{1}(1/l) & \mp n_1 & && \\
%\vdots & \vdots & \vdots& \vdots & \vdots & \vdots \\
%L_{2g+h+l} & T_{l}(l/l) & \mp n_l & && \\
%\end{array}
%$$

For a knot $K(t)$, we have
$$\tb_{S^3}\big(K(t)\big) = \lk \big(K(t), K(t+\varepsilon)\big)$$
and hence, 
for $i=1,\ldots,l$,
$$
\tb_{S^3}(L_{2g+h+i})
=
-\sum_{k=1}^{g+h}\big( (T_i\bullet a_k)^2 + (T_i \bullet a_k)(T_i\bullet b_k)
+ (T_i\bullet b_k)^2 \big).
$$
Therefore, the topological surgery coefficient of $L_{2g+h+i}$ is
$$
\frac{p_{2g+h+i}}{q_{2g+h+i}} = \frac{n_i\tb_{S^3}(T_i)\mp 1}{n_i}.
$$
Furthermore,
the linking behaviour with $L_j = \beta_j$, $j=1,\ldots,g+h$ is
$$
\lk (L_{2g+h+i}, L_j) = - (T_i\bullet b_j)
$$
and similarly, for $L_{g+h+j} = \alpha_j$, $j=1,\ldots,g$
$$
\lk (L_{2g+h+i}, L_{g+h+j}) = - \big( (T_i\bullet a_j) + (T_i\bullet b_j) \big).
$$
The linking number of two surgery knots $L_{2g+h+i}$ and $L_{2g+h+j}$ with $i<j$ can be
computed to be
$$
\lk (L_{2g+h+i}, L_{2g+h+j})
=
-\sum_{k=1}^{g+h}\big( (T_i\bullet a_k)(T_j\bullet a_k)
+ (T_i \bullet a_k)(T_j\bullet b_k)
+ (T_i\bullet b_k)(T_j\bullet b_k) \big).
$$

Note that the knot $K$ can be put on the page with the lowest as well as the highest
level.
Depending on which is chosen, the class of Seifert surface with respect to which the
rotation number is given in Remark~\ref{rem:rot_surgery} might change, and hence
the rotation numbers may differ.
However, if the Euler class of $\xi$ vanishes, the rotation number of a nullhomologous is independent
of the Seifert surface.
If we choose the knot $L_0 = K(\textrm{low})$ to sit on a lower page than the surgery link, we get
the following linking numbers
$$
\arraycolsep=5pt\def\arraystretch{1.6}
\begin{array}{rl}
\lk (L_0, L_j) = & - \big( (K\bullet a_j) + (K\bullet b_j) \big),\; j=1,\ldots,g+h,\\
\lk (L_0, L_{g+h+j}) = & - (K\bullet a_j),\; j=1,\ldots,g,\\
\lk (L_0, L_{2g+h+j})
=&
-\displaystyle\sum_{k=1}^{g+h}\big( (K\bullet a_k)(T_j\bullet a_k)
+ (K\bullet a_k)(T_j\bullet b_k)\\&
+ (K\bullet b_k)(T_j\bullet b_k) \big),\; $j=1,\ldots,l$.
\end{array}
$$

If on the other hand $L_0 = K(\textrm{high})$ is assumed to sit on a page with
the highest level, we get
$$
\arraycolsep=5pt\def\arraystretch{1.6}
\begin{array}{rl}
\lk (L_0, L_j) = &- (K\bullet b_j),\; j=1,\ldots,g+h,\\
\lk (L_0, L_{g+h+j}) = & - \big( (K\bullet a_j) + (K\bullet b_j) \big),\; j=1,\ldots,g,\\
\lk (L_0, L_{2g+h+j})
=&
-\displaystyle\sum_{k=1}^{g+h}\big( (K\bullet a_k)(T_j\bullet a_k)
+ (K\bullet b_k)(T_j\bullet a_k)\\&
+ (K\bullet b_k)(T_j\bullet b_k) \big),\; j=1,\ldots,l.
\end{array}
$$

The only data that is left to compute are the rotation numbers in $S^3$ of the $L_i$,
but this can be done as in Proposition~\ref{prop:special_s3}.
Observe that using the formula from \cite{Kegel2016} also allows us to calculate
the Thurston--Bennequin invariant, which is an alternative to the method presented
in \cite{tb_openbooks}.
Similarly, one can directly calculate the
Poincar\'{e}-dual of the Euler class and the $d_3$-invariant of the contact structure
(see \cite[Theorem~5.1]{rot_surgery}).

Thus, we have proved Theorem~\ref{thm:rot_in_ob}.
\end{proof}

\section{Algorithm and examples}
\label{section:algorithm}

We summarise the process and all required formulas in the following algorithm and
illustrate them by giving examples. This section is meant as a self-contained
guideline to do actual computations and can be used independently from the rest of the paper.

\begin{algorithm}
\label{algo_everything}
\emph{The setting.}\\
Given is a non-isolating curve $K$ on the page of an open book
$$
(\Sigma_{g,h+1}, \;\phi = T_l^{\pm n_l} \circ\cdots\circ T_1^{\pm n_1})
$$
with $n_i\in \N$ and $\Sigma_{g,h+1}$ a surface of genus $g$ with $h+1$
boundary components. The monodromy is given as a sequence of Dehn twists along
non-isolating oriented curves~$T_i$.

\vspace{0.5em}
\noindent
\emph{The choices.}\\
Choose reducing arcs $r_1,\ldots, r_{g+h-1}$ such that
when cutting along~$r_i$
\begin{itemize}
\item
$\Sigma$ decomposes into a surface $\Sigma_i$ of genus $i$
with one boundary component
containing $r_1,\ldots, r_{i-1}$
and a surface of genus $g-i$ with $h+1$ boundary components
for $i=1,\ldots,g$,
\item
$\Sigma$ decomposes into a surface $\Sigma_i$ of genus $g$ with $i+1$ boundary components
containing $r_1,\ldots, r_{i-1}$
and a disk with $h-i$ holes
for $i=g+1,\ldots,g+h-1$.
\end{itemize}
Then choose an arc basis of $\Sigma_i\setminus\Sigma_{i-1}$ and label it by $a_i, b_i$
and orient it
such that when travelling along the oriented boundary of $\Sigma$ from 
\begin{itemize}
\item $r_1$ to $r_1$
\begin{itemize}
\item first $a_1$ is met pointing outwards, then $b_1$ is met pointing inwards
if $g\geq 1$
\item
$b_1$ is met and pointing outwards if $g=0$
\end{itemize}
\item $r_{i-1}$ to $r_i$ only $b_i$ is met and pointing outwards ($i=2,\ldots,g+h-2$)
\item
$r_{g+h-1}$ to $r_{g+h-1}$
\begin{itemize}
\item first $b_g$ is met pointing outwards, then $a_g$ is met pointing outwards
if $h=0$
\item $b_{g+h}$ is met and pointing outwards if $h>0$.
\end{itemize}
\end{itemize}
Choose non-trivial oriented simple closed curves
$\alpha_i,\beta_i$ representing a basis of $H_1(\Sigma)$ dual to the arcs
with respect to the intersection product on $\Sigma$
oriented such that $\alpha_i\bullet a_i = 1$, $\beta_j\bullet b_j=1$ and
$\alpha_i\bullet\beta_i = 1$
(i.e.\ the situation is as in Figure~\ref{fig:AbstractOpenBookWithBoundary}).

\vspace{0.5em}
\noindent
\emph{The definitions.}\\
Define an integral vector $\mathbf{l}\in \Z^{2g+h+l}$ with entries:
$$
\arraycolsep=5pt\def\arraystretch{1.5}
\begin{array}{rl}
%\multicolumn{2}{l}{
%\textsl{for }j=1,\ldots,g+h:}\\
\mathbf{l}_j = &
- (K\bullet b_j),\\
&\textsl{for }j=1,\ldots,g+h,\\
%
%\multicolumn{2}{l}{
%\textsl{for }j=1,\ldots,g:}\\
\mathbf{l}_{g+h+j} = & - \big( (K\bullet a_j) + (K\bullet b_j) \big),\\
&\textsl{for }j=1,\ldots,g,\\
%
%\multicolumn{2}{l}{
%\textsl{for }j=1,\ldots,l:}\\
\mathbf{l}_{2g+h+j}
=&
-\displaystyle\sum_{k=1}^{g+h}\big( (K\bullet a_k)(T_j\bullet a_k)
+ (K\bullet b_k)(T_j\bullet a_k)\\
&
+ (K\bullet b_k)(T_j\bullet b_k) \big),\\
&\textsl{for }j=1,\ldots,l,\\
\end{array}
$$
Define an integral $({2g+h+l})\times({2g+h+l})$-matrix $Q$ with entries:
$$
\arraycolsep=5pt\def\arraystretch{1.5}
\begin{array}{rl}
Q_{i,j} =& 0,\\
&\textsl{for } i,j=1,\ldots,2g+h,\\
%
%Q_{i,i+g+h} =& -1 = Q_{j,j-g-h},\\
%&\textsl{for }i=1,\ldots,g,\; j=g+h+1,\ldots,2g+h,\\
%
Q_{2g+h+i,2g+h+i} = &\mp 1 - n_i \displaystyle\sum_{k=1}^{g+h}\big( (T_i\bullet a_k)^2 + (T_i \bullet a_k)(T_i\bullet b_k)
+ (T_i\bullet b_k)^2 \big),\\
&\textsl{for }i=1,\ldots,l,\\
Q_{2g+h+i,j}= & - (T_i\bullet b_j),\\
&\textsl{for }i=1,\ldots,l,\; j=1,\ldots,g+h,\\
Q_{2g+h+i,j} = &
- \big( (T_i\bullet a_j) + (T_i\bullet b_j) \big),\\
&\textsl{for }i=1,\ldots,l,\; j=g+h+1,\ldots,2g+h,
\end{array}
$$
$$
\arraycolsep=5pt\def\arraystretch{1.5}
\begin{array}{rl}
Q_{i,2g+h+j} = &
- n_j(T_j\bullet b_i),\\
&\textsl{for }i=1,\ldots,g+h,\; j=1,\ldots,l,\\
Q_{g+h+i,2g+h+j} = &
- n_j\big( (T_j\bullet a_i) + (T_j\bullet b_i) \big),\\
&\textsl{for }i=1,\ldots,g,\; j=1,\ldots,l,\\
Q_{2g+h+i,2g+h+j} = &
-n_j\displaystyle\sum_{k=1}^{g+h}\big( (T_i\bullet a_k)(T_j\bullet a_k)
+ (T_i \bullet a_k)(T_j\bullet b_k)\\
&
+ (T_i\bullet b_k)(T_j\bullet b_k) \big),\textsl{for }i<j, \; i,j=1,\ldots,l,\\
Q_{2g+h+i,2g+h+j} = &
-n_i\displaystyle\sum_{k=1}^{g+h}\big( (T_i\bullet a_k)(T_j\bullet a_k)
+ (T_i \bullet b_k)(T_j\bullet a_k)\\
&
+ (T_i\bullet b_k)(T_j\bullet b_k) \big),\textsl{for }i>j, \; i,j=1,\ldots,l.
\end{array}
$$
For an oriented non-isolating curve $L$ we define the quantity $r(L)$ as follows:
choose a starting point on $L$ and write $L$ as a word in the $\alpha_i$ and $\beta_i$
by noting intersections with $a_i$ and $b_i$ when traversing along $L$.
Set $\lambda_+$ to be the number of times a $\beta^{-1}$ is followed by an
$\alpha^{-1}$ of the same index also considering the step from the last to the first letter,
and similarly, set $\rho_+$ equal to the number of times an $\alpha^{-1}$ is
followed by a $\beta^{-1}$ of the same index.
Denote places where the index changes by $r_u$ ($r_d$) if the index
increases (decreases) -- including the last position if the index of the last
letter is not equal to the index of the first letter.
Now run through the index changes and increase $\lambda_+$ and $\rho_+$
according to the following rule:

\begin{itemize}
\item increase $\lambda_+$ by $1$ for
\begin{itemize}
\item a $\beta^{-1}$ followed by $r_u$
\item $r_d$ followed by an $\alpha^{-1}$
\end{itemize}
\item increase $\rho_+$ by $1$ for
\begin{itemize}
\item an $\alpha^{-1}$ followed by $r_d$
\item a $\beta$ followed by $r_d$.
\end{itemize}
\end{itemize}
Then define
$$
r(L) := \rho_+ - \lambda_+.
$$

\vspace{0.5em}
\noindent
\emph{The results.}\\
Then the following holds:
\begin{itemize}
\item[(a)]
$K$ is nullhomologous if and only if
there is an integral solution
$\mathbf{a}$ of the equation $\mathbf{l} = Q\mathbf{a}$.
\item[(a')]
$K$ is rationally nullhomologous in the manifold if and only if
there is a rational solution
$\mathbf{a}$ of the equation $\mathbf{l} = Q\mathbf{a}$.
\item[(b1)]
If $K$ is (rationally) nullhomologous, the (rational) Thurston--Bennequin invariant of
$K$ is
\begin{align*}
\tb(K) = &
-\sum_{k=1}^{g+h}\big( (K\bullet a_k)^2 + (K \bullet a_k)(K\bullet b_k)
+ (K\bullet b_k)^2 \big) \\
& - \sum_{j=1}^{2g+h}{a_j l_{j}}
 - \sum_{j=1}^{l}{a_{2g+h+j} n_j l_{2g+h+j}}.
\end{align*}
\item[(b2)]
If $K$ is (rationally) nullhomologous, the (rational) rotation number with respect to
some special Seifert surface $S$ of $K$ is
$$
\rot (K, S) = r(K)- \sum_{j=1}^{l}{a_{2g+h+j} n_j r(T_j)}.
$$
\item[(b3)]
Denote by $K^\pm$ the positive (resp. negative) transverse push-off of a (rationally) nullhomologous Legendrian $K$.
Then its (rational) self-linking number with respect to the Seifert surface $S$
from (b2) is
$$
\selfl(K^\pm, S) = \tb(K) \mp \rot(K,S).
$$
\item[(c)]
The Poincar\'{e}-dual of the Euler class is given by
\begin{equation*}
\operatorname{PD}\big(\textrm{e}(\xi)\big)=\sum_{i=1}^l n_i r(T_i)\mu_{T_i}\in H_1(M).
\end{equation*}
The first homology group
$H_1(M)$ of $M$ is generated by the meridians $\mu$
of the $\alpha_i, \beta_i$ and $T_i$
and the relations are given by the generalized linking matrix $Q\mathbf\mu=0$.

\item[(d)]
The Euler class $\textrm{e}(\xi)$ is torsion if and only if there exists
a rational solution $\mathbf b$ of $Q\mathbf b=\mathbf{r}$
with $\mathbf{r}_i = 0$ for $i=1,\ldots, 2g+h$ and
$\mathbf{r}_{2g+h+i} = r(T_i)$ for $i=1,\ldots, l$.
In this case, the $d_3$-invariant of $\xi$ computes as
\begin{equation*}
d_3(\xi) =
g + \frac{h}{2} +
\frac{1}{4} \left(\sum_{i=1}^l n_i b_{2g+h+i} r(T_i)  - (3-n_i) \operatorname{sign}_i\right)   - \frac{3}{4} \sigma (Q) - \frac{1}{2} ,
\end{equation*}
where $\operatorname{sign}_i$ denotes the sign of the power of the Dehn twist $T_i^{\pm n_i}$.
\end{itemize}
\end{algorithm}

\begin{remark}[]
In the algorithm above, we implicitly assumed that the knot $K$ sits on the page
with a higher level than the monodromy curves.
As described in Section~\ref{section:general_case}, $K$ could also be assumed to sit on the lowest level, which
would change the formulas defining the vector $\mathbf{l}$.
Note that in general, if $e(\xi)\neq 0$, the resulting rotation number might differ,
as it is computed with respect to a different class of Seifert surface.
However, if the open book is planar or $e(\xi) = 0$, we get the same values for both cases.
\end{remark}

\begin{remark}
\label{rem:planar_easy_formula}
In the planar case, the formulas simplify to
$$
\arraycolsep=5pt\def\arraystretch{1.6}
\begin{array}{rl}
%\multicolumn{2}{l}{
%\textsl{for }j=1,\ldots,g+h:}\\
\mathbf{l}_j = &
- (K\bullet b_j),\;\textsl{for }j=1,\ldots,h,\\
%
%\multicolumn{2}{l}{
%\textsl{for }j=1,\ldots,l:}\\
\mathbf{l}_{h+j}
=&
-\displaystyle\sum_{k=1}^{h}(K\bullet b_k)(T_j\bullet b_k),\;\textsl{for }j=1,\ldots,l,\\
Q_{i,j} =& 0,\;\textsl{for } i,j=1,\ldots,h,
\end{array}
$$
$$
\arraycolsep=5pt\def\arraystretch{1.6}
\begin{array}{rl}
Q_{h+i,h+i} = &\mp 1 - n_i \displaystyle\sum_{k=1}^{h}(T_i\bullet b_k)^2,\;\textsl{for }i=1,\ldots,l,\\
Q_{h+i,j}= & - (T_i\bullet b_j),\;\textsl{for }i=1,\ldots,l,\; j=1,\ldots,h,\\
Q_{i,h+j} = &
- n_j(T_j\bullet b_i),\;\textsl{for }i=1,\ldots,h,\; j=1,\ldots,l,\\
Q_{h+i,h+j} = &
-n_j\displaystyle\sum_{k=1}^{h}(T_i\bullet b_k)(T_j\bullet b_k),\;\textsl{for }i\neq j, \; i,j=1,\ldots,l.
\end{array}
$$
If furthermore all $n_i = 1$, we have that

$$
Q =
\begin{pmatrix}
Q_1 & Q_2 \\
Q_3 & Q_4
\end{pmatrix}
$$
with
$Q_1 = 0_{h\times h}$ the zero ($h\times h$)-matrix,
$$
Q_2 = Q_3^t = -\big( T_j\bullet b_i\big)_{i=1,\ldots,h;\; j=1,\ldots,l}
$$
and 
$$
Q_4 = Q_3 Q_2 \mp \textrm{diag}\big(\sign(T_1),\ldots,\sign(T_l)\big).
$$
\end{remark}

\begin{example}
\label{ex:planar_word}
In this example we want to reconsider the planar open book of $(S^3,\xist)$ discussed in Example~\ref{ex:planar}, where we calculated the rotation number to be $2$
using Proposition~\ref{prop:planar}.
If we choose the arc basis as described above, the knot is encoded by the word
$\beta_2 \beta_4 \beta_3$.
This yields $\lambda_+ = 0$ and $\rho_+ = 2$, i.e.\ $\rot = 2$ as expected.

Note that if we choose a different arc basis, e.g.\ such that the word is
$\beta_2 \beta_3 \beta_4$, then the formula does not give the desired result, as
the knot would be represented by a different word. In fact, the word
$\beta_2 \beta_3 \beta_4$ does not even encode a simple closed curve on the embedded
page.
\end{example}

\begin{figure}[htb] 
\centering
\def\svgwidth{\columnwidth}
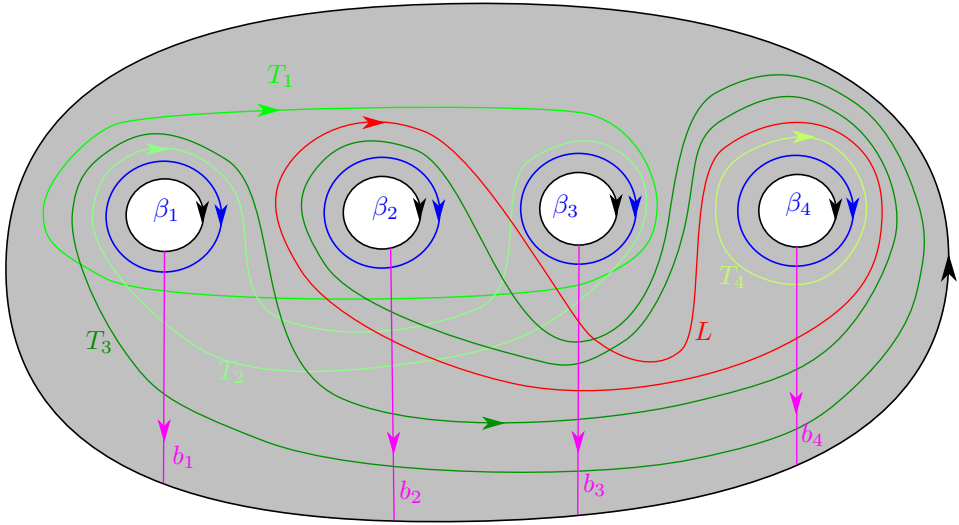
\caption{The open book $(\Sigma, \phi = T_3^{+1}\circ T_2^{+1}\circ T_1^{+1})$ of $(S^3,\xist)$.}
\label{fig:PlanarEt}
\end{figure}

\begin{example}
Consider the open book $(\Sigma, \phi = T_3^{+1}\circ T_2^{+1}\circ T_1^{+1})$
and knot $L$ as specified in Figure~\ref{fig:PlanarEt}.
This is an example of a non-destabilisable planar open book of $(S^3,\xist)$
taken from~\cite{EtLi15}.

By the formulas to compute $\rot$ in the special planar case,
it follows directly that
$$
\mathbf{r} = (0, 0, 0, 0, 2, 1, 1, 0)^t
$$
and $r(L) = 1$.

Using the simplified formulas for planar open books given in
Remark~\ref{rem:planar_easy_formula},
we obtain
$$
\mathbf{l} = (0, -1, 0, -1, -1, 0, -1, -1)^t
$$
and
$$
Q_2 =
-\begin{pmatrix}
1 & 1 & 1 & 0\\
1 & 0 & 1 & 0\\
1 & 1 & 0 & 0\\
0 & 0 & 0 & 1
\end{pmatrix}.
$$
As the manifold is $S^3$, it follows that $Q$ is invertible and thus the equation $
\mathbf{l} = Q\mathbf{a}$ admits a unique solution, which is easily computed to be
$$
\mathbf{a} = (2, -2, -1, -1, 1, -1, 0, 1)^t
$$
(in particular, the calculation shows that $L$ is nullhomologous).
The Thurston--Bennequin invariant of $L$ then computes to be
$$
\tb (L) = -\sum_{k=1}^{4}(L\bullet b_k)^2 - \langle \mathbf{a}, \mathbf{l}\rangle
= -3
$$
and the rotation number is
$$
\rot (L) = r(L) - \langle \mathbf{a}, \mathbf{r}\rangle = 0.
$$
The self-linking number of both the positive and the negative transverse push-off
of $L$ is $-3$.

Since $Q$ is invertible, we have $H_1 = 0$, i.e.\ the Poincar\'e dual to the Euler class of the contact structure $\xi$ vanishes.
As expected, our formula then returns
$$
d_3 (\xi) = -\frac{1}{2}.
$$
\end{example}

\section*{Acknowledgements}
The results in this paper are also contained in the thesis of the first author~\cite{KoelschDurst}. We would like to thank our advisor Hansj\"org Geiges and our former colleagues Mirko Klukas and Thomas Rot for useful discussions. Moreover, we would like to thank Youlin Li for pointing out a mistake in Algorithm~\ref{algo_everything} and Example~\ref{ex:nonempty} in an earlier version.

%%%%%%%%%%%%%%%%%%%%%%%%%%%%%%%%%%%%%%%%%%%%%%%%%%%%%%%%%%%%%%%%%%%%%%%%%%%%%%%%

%%%%%%%%%%%%%%%%%%%%%%%%%%%%%%%%%%%%%%%%%%%%%%%%%%%%%%%%%%%%%%%%%%%%%%%%%%%%%%%%

\end{document}